\documentclass[a4paper, 11pt, underruledhead]{paper}
\usepackage{amssymb}
\usepackage{amsmath}
\usepackage{amsthm}
\usepackage[mathscr]{eucal}
\usepackage{enumerate}
\usepackage{graphics}
\begin{document}



\let\goth\mathfrak


\newcommand\theoremname{Theorem}
\newcommand\lemmaname{Lemma}
\newcommand\corollaryname{Corollary}
\newcommand\propositionname{Proposition}
\newcommand\factname{Fact}
\newcommand\remarkname{Remark}
\newcommand\examplename{Example}

\newtheorem{thm}{\theoremname}[section]
\newtheorem{lem}[thm]{\lemmaname}
\newtheorem{cor}[thm]{\corollaryname}
\newtheorem{prop}[thm]{\propositionname}
\newtheorem{fact}[thm]{\factname}
\newtheorem{exmx}[thm]{\examplename}
\newenvironment{exm}{\begin{exmx}\normalfont}{\end{exmx}}

\newtheorem{rem}[thm]{\remarkname}

\newtheorem*{thm*}{\theoremname}

\def\myend{{}\hfill{\small$\bigcirc$}}

\newtheorem{reprx}[thm]{Representation}
\newenvironment{repr}{\begin{reprx}\normalfont}{\myend\end{reprx}}
\newtheorem{cnstrx}[thm]{Construction}
\newenvironment{constr}{\begin{cnstrx}\normalfont}{\myend\end{cnstrx}}
\def\classifname{Classification}
\newtheorem{classification}[thm]{\classifname}
\newenvironment{classif}{\begin{classification}\normalfont}{\myend\end{classification}}

\newenvironment{ctext}{%
  \par
  \smallskip
  \centering
}{%
 \par
 \smallskip
 \csname @endpetrue\endcsname
}


\newcounter{sentence}
\def\thesentence{\roman{sentence}}
\def\labelsentence{\upshape(\thesentence)}

\newenvironment{sentences}{%
        \list{\labelsentence}
          {\usecounter{sentence}\def\makelabel##1{\hss\llap{##1}}
            \topsep3pt\leftmargin0pt\itemindent40pt\labelsep8pt}%
  }{%
    \endlist}

\newcounter{typek}\setcounter{typek}{0}
\def\thetypek{\roman{typek}}
\def\typitem#1{\refstepcounter{typek}\item[\normalfont(\thetypek;\hspace{1ex}{#1})\quad]}
\newenvironment{typcription}{\begin{description}\itemsep-2pt\setcounter{typek}{0}}{%
\end{description}}


\newcommand*{\sub}{\raise.5ex\hbox{\ensuremath{\wp}}}
\newcommand{\msub}{\mbox{\large$\goth y$}}    
\def\id{\mathrm{id}}
\def\Aut{\mathrm{Aut}}
\newcommand*{\struct}[1]{{\ensuremath{\langle #1 \rangle}}}

\def\kros(#1,#2){c_{\{ #1,#2 \}}}
\def\skros(#1,#2){{\{ #1,#2 \}}}
\def\inc{\mathrel{\strut\rule{3pt}{0pt}\rule{1pt}{9pt}\rule{3pt}{0pt}\strut}}
\def\lines{{\cal L}}
\def\tran{{\mathscr T}}
\def\collin{\sim}
\def\ncollin{\not\collin}

\def\VerSpace(#1,#2){{\bf V}_{{#2}}({#1})}
\def\GrasSpace(#1,#2){{\bf G}_{{#2}}({#1})}
\def\VeblSpSymb{{\bf P}\mkern-10mu{\bf B}}
\def\VeblSpace(#1){\VeblSpSymb({#1})}
\def\xwpis#1#2{{{\triangleright}\mkern-4mu{{\strut}_{{#1}}^{{#2}}}}}
\def\MultVeblSymb{{\sf M}\mkern-10mu{\sf V}}        
\def\MultVeblSymq{{\sf T}\mkern-6mu{\sf P}}        
\def\xwlep(#1,#2,#3,#4){{\MultVeblSymb^{#1}{\xwpis{#2}{#3}}{#4}}}
\def\xwlepp(#1,#2,#3,#4,#5){{\MultVeblSymb_{#1}^{#2}{\xwpis{#3}{#4}}{#5}}}
\def\xwlepq(#1,#2,#3,#4,#5){{\MultVeblSymq_{#1}^{#2}{\xwpis{#3}{#4}}{#5}}}
\def\konftyp(#1,#2,#3,#4){\left( {#1}_{#2}\, {#3}_{#4} \right)}

\newcount\liczbaa
\newcount\liczbab
\def\binkonf(#1,#2){\liczbaa=#2 \liczbab=#2 \advance\liczbab by -2
\def\doa{\ifnum\liczbaa = 0\relax \else
\ifnum\liczbaa < 0 \the\liczbaa \else +\the\liczbaa\fi\fi}
\def\dob{\ifnum\liczbab = 0\relax \else
\ifnum\liczbab < 0 \the\liczbab \else +\the\liczbab\fi\fi}
\konftyp(\binom{#1\doa}{2},#1\dob,\binom{#1\doa}{3},3) }


\def\Des{{\sf\bfseries DES}}
\def\Desv{{\sf\bfseries DES'}}
\def\Desr{{\sf\bfseries DES''}}
\def\MTC{{\sf STP}}
\def\MVC{{\sf MVC}}


\title[Classification of $\konftyp(15,4,20,3)$-configurations with three $K_5$-graphs]{%
A complete classification of the $\konftyp(15,4,20,3)$-configurations with at least three
$K_5$-graphs}


\author{K. Petelczyc, M. Pra{\.z}mowska, K. Pra{\.z}mowski}

\maketitle

\begin{abstract}
  The class of  $\binkonf(n,+1)$-configurations
  which contain at least $n-2$ $K_n$-graphs coincides with 
  the class of so called systems of triangle perspectives i.e. 
  of configurations which contain a bundle of $n-2$ Pasch configurations with a common line.
  For $n=5$ the class consists of all
  binomial partial Steiner triple systems on $15$ points,
  that contain at least three $K_5$-graphs.
  In this case a complete classification of respective configurations is given and
  their automorphisms are determined.
\end{abstract}

\begin{flushleft}\small
  Mathematics Subject Classification (2010): 05B30 (05C51, 05B40).\\
  Keywords: binomial partial Steiner triple system, complete graph contained in a configuration,
  Veblen (Pasch) configuration, (generalized) Desargues configuration,
  $10_3G$ Kantor's configuration, multiveblen configuration, system of triangle perspectives.
\end{flushleft}

\section*{Introduction}

It is a quite common 
research project to characterize (and classify)
configurations (more generally: block designs) which contain (or not) 
subconfigurations in a definite class
(comp. \cite{bezpasz}, \cite{bezkwadrat}. \cite{bezmitra}).
In the case of our paper these are: partial Steiner triple systems with complete
graphs `inside'.

The minimal size of a partial Steiner triple system i.e. of a 
$(v_r \, b_3)$-configuration which contains a complete graph $K_n$ is
$v = \binom{n+1}{2}$, $b = \binom{n+1}{3}$, $r = n-1$, so such a minimal configuration
${\sf PSTS}$ is a binomial configuration.

Generally, a binomial $\binkonf(n,+1)$-configuration may contain
$0,1,\ldots,n-1,n+1$ graphs $K_n$.
All the (minimal) configurations which  contain $K_4$ were classified in 
\cite{klik:VC}. These all are well known $10_3$-configurations
(comp. \cite{betten}, \cite{obrazki}).

Binomial (minimal) configurations with $K_5$ are
$\konftyp(15,4,20,3)$-configurations, so each of them may contain
$0,1,2,3,4$ or $6$ $K_5$-graphs. In this paper we classify  all these
$\konftyp(15,4,20,3)$-configurations which contain at least three copies
of $K_5$.
We prove  that there are seventeen such configurations
(four of them are so called multi-Veblen configurations \MVC\ with
$4$  or $6$ copies of $K_5$),
we present each of these configurations,
and we determine the automorphism group of each of them.

As a technical tool to achieve  our  result we use the representation
of our $\konftyp(15,4,20,3)$-configurations as so-called systems of 
triangle perspectives \MTC.

In accordance with the general theory every binomial 
$\binkonf(n,+1)$-con\-fi\-gu\-ra\-tion 
with at least $n-2$ copies of $K_n$ can 
be represented as a system of triangle perspectives \MTC\ (comp. \cite{klik:binom}).
The number of \MTC's grows rapidly; 
for $n = 4$ there are three \MTC's, 
two of them are \MVC, 
for $n =5$ there are seventeen \MTC's, 
three of them are (simple) \MVC, 
for $n =6$
there are seven simple \MVC's (cf. \cite{pascvebl}) 
and at least thirty other \MTC's (the classification is not yet completed). 
The problem  to classify all the minimal $\sf PSTS$'s with at least
$n-2$ $K_n$-graphs in each is, theoretically, solved: it 
is equivalent to classification of 
all the edge-colorings
of the complete digraph on $n-2$ vertices 
by the elements of the group $S_3$.
Practically, for $n > 5$ this classification cannot be achieved, we think,
`by hand' and computer methods must be used.


\section{Definitions and the results}


\subsection{Some (basic) combinatorial configurations}
Recall that the term {\em combinatorial configuration} 
(or simply a configuration) is usually (cf. \cite{grop1}, \cite{steiniz}, \cite{levi}) 
used for a finite incidence point-line structure with constant line size and point rank
provided that two different points are 
incident with at most one line.
A {\em $(v_r,b_k)$-configuration} is a combinatorial configuration 
with $v$ points and $b$ lines such that
there are $r$ lines through each point, and there are $k$ points on each line.
A {\em binomial configuration} is a combinatorial $\konftyp(v,r,b,\kappa)$-configuration
such that $v = \binom{r + \kappa - 1}{r}$ and $b = \binom{r + \kappa -1}{\kappa}$.
In what follows we shall be primarily interested in 
{\em binomial partial Steiner triple systems} 
i.e. in the $\binkonf(n,1)$-configurations with arbitrary 
integer $n\geq 3$.

%

Let $k$ be a positive integer and $X$ a set; we write $\sub_k(X)$ for the 
set of all $k$-subsets of $X$.
A multiset with repetitions ({\em a multiset}) of cardinality $k$ with elements in the set
$X$ is a function $f \colon X \longrightarrow N$ such that 
  $|f| := \sum_{x \in X} f(x) = k < \infty$. 
We write $\msub_k(X)$ for the family of 
all such multisets. Clearly, if $f \in \msub_k(X)$ then 
  ${\mathrm{supp}}(f) = \{ x \in X \colon f(x)\neq 0 \}$ 
is finite. 
In particular, if $X$ is finite, we can identify 
$f$ with the (formal) polynomial 
%
   $f = \prod_{x \in {\mathrm{supp}}(f)}  x^{f(x)} = \prod_{x \in X}  x^{f(x)}$ 
%
with variables in $X$; 
we have $|\msub_k(X)| = \binom{|X| + k -1}{k}$.

The incidence structure
\begin{ctext}
  $\GrasSpace(X,k) := \struct{\sub_k(X),\sub_{k+1}(X),\subset}$
\end{ctext}
will be called {\em a combinatorial Grassmannian} (cf. \cite{perspect}).

The ($m$-th) {\em combinatorial Veronesian} over $X$
is the incidence structure
\begin{ctext}
  $\VerSpace(X,m) := \struct{\msub_m(X),{\cal L}^\ast,\in}$,
\end{ctext}
where
  $ {\cal L}^\ast  = 
  \big{\{} e  X^r  
  \colon 1 \leq r \leq m, \;  e \in \msub_{m-r}(X) \big{\}}$
and 
  $e X^r = \{ e x^r \colon x\in X \}$
(cf. \cite{combver}).

\par
Let $X$ be a nonempty set and $|X| = n$;
we write
  $\VerSpace(n,m) = \VerSpace(X,m)$
and
  $\GrasSpace(n,k) = \GrasSpace(X,k)$.
%
Let us quote three classical examples:
  $\GrasSpace(4,2) \cong \VerSpace(3,2)$ is the Veblen (Pasch) configuration,
  $\GrasSpace(5,2)$ is the Desargues configuration \Des, and
  $\VerSpace(3,3)$ is the  ${10}_{3}G$ Kantor configuration
(see \cite{combver}, \cite{kantor}).
The last one is sometimes called the Veronese configuration (\cite{combver}). 

Let $X$ be a non void set. Formally, 
{\em a nondirected loopless graph defined on $X$} is 
a structure of the form $\struct{X,{\cal P}}$ with ${\cal P}\subset \sub_2(X)$ 
but in the sequel we shall
frequently identify it with $\cal P$ and we shall call $\cal P$ simply a graph.
We write $K_X$ and $N_X$ for the complete graph with the vertices $X$ and for 
its complement, respectively.

Finally, let us recall one of the definitions of a multiveblen configuration 
(cf. \cite{pascvebl}).
Let $\cal P$ be a graph defined on $X$, $|X| = n$, and let 
${\goth H} = \struct{\sub_2(X),\lines}$ be a partial Steiner triple system.
The {\em points} of 
{\em the multiveblen configuration 
${\goth M} := \xwlepp({X},{p},{\cal P},{},{\goth H})$} 
are the following:
$p$, $a_i,b_i$ with $i \in X$, $c_z$ with $z \in \sub_2(X)$.
The {\em lines}  of $\goth M$ are the following sets:
\begin{ctext}
  $\{ a_i,b_j,c_{\{ i,j \}} \}$ and
  $\{ a_j,b_i,c_{\{ i,j \}} \}$ for $\{i,j\}\in\sub_2(X)\setminus{\cal P}$;
  \\
  $\{ a_i,a_j,c_{\{ i,j \}} \}$ and
  $\{ b_i,b_j,c_{\{ i,j \}} \}$ for $\{i,j\}\in{\cal P}$;
  \\
  $\{ p,a_i,b_i \}$, $i\in X$;\quad 
  $\{ c_u,c_v,c_w \}$, where $\{ u,v,w \}$ is a line of $\goth H$.
\end{ctext}
A multiveblen configuration is {\em simple} if ${\goth H} = \GrasSpace(X,2)$.
The point $p$ is a {\em center} of $\goth M$.
\begin{fact}
  $\VerSpace(k,m)$ and $\GrasSpace(n,k-1)$ are 
  $\konftyp({\binom{n}{m}},{m},{\binom{n}{k}},{k})$-configurations,
  where ${n = m+k-1}$.

  If $\goth H$ is a $\konftyp({\binom{n}{2}},{n-2},{\binom{n}{3}},{3})$-configuration then
  $\xwlepp({X},{p},{\cal P},{},{\goth H})$ is a 
  $\konftyp(\binom{n+2}{2},{n},\binom{n+2}{3},{3})$-configuration.
 \par
  For every distinct $i,j\in X$ the substructure of \ 
  $\xwlepp({X},{p},{\cal P},{},{\goth H})$
  with the points 
    $p,a_i,a_j,b_i,b_j,c_{\{i,j\}}$ 
  is the Veblen configuration.
 \par
  If $|X| = n$ then 
  $\xwlepp({X},{p},{K_X},{},{\GrasSpace(X,2)}) \cong \GrasSpace(n+2,2)$.
\end{fact}
Note that combinatorial Grassmannians $\GrasSpace(n,2)$, 
combinatorial Veronesians $\VerSpace(3,k)$, 
and multiveblen 
configurations defined above are binomial partial Steiner triples.

We say that a configuration ${\goth M} = \struct{S,\lines}$ 
{\em freely contains} the complete graph $K_X$ if and only if $X \subset S$,
for every edge $e\in\sub_2(X)$ there is a unique block $\overline{e}\in\lines$
that contains $e$, 
and any two edges $e_1,e_2\in\sub_2(X)$ such that 
$\overline{e_1}\cap\overline{e_2} \neq\emptyset$ have a common vertex in $X$.

\subsection{Systems of triangle perspectives}\label{ssec:STP}

In the paper we shall consider also configurations defined by the following 
construction,
which proposes another approach to ``gluing" Veblen configurations, 
more general than the notion of a multiveblen configuration.
\begin{constr}\label{cnstr:MTC}
  Let $I$ be an $n$-element set,
  $\tran :=  \{ a,b,c \}$, and $I \cap \tran = \emptyset$. 
  Moreover, adopt the convention $X \in \{ A,B,C  \}$ ($X=A$ if $x =a$ etc.,
  where $x \in \tran$).
  Let $n$ Veblen configurations ${\mathscr V}^i$ labelled by the elements $i\in I$,
  \begin{ctext}
    ${\mathscr V}^i = 
    \struct{\{ q^a,q^b,q^c,a_i,b_i,c_i \}, \{ L,A_i,B_i,C_i \}}$, $i \in I$,
  \end{ctext}
  have a common line
  $L = \{ q^a,q^b,q^c \}$, let
  $q^x \inc X_i$, $x_i \not\inc X_i$ for $i \in I$, $x \in \tran$.
  Let perspectives%
  \footnote{%
  The term ``perspective" used here may be slightly misleading; $\xi^{i,j}$ is not any
  ``real" perspective of ${\mathscr V}^i$ onto ${\mathscr V}^j$, as the 
  latter should fix the points on $L$, which yields $\xi(i,j)=\id$.
  Therefore, in the sequel we prefer to use the term ``perspective of triangles".
  }
  ${\xi}^{i,j}\colon {\mathscr V}^i \longrightarrow{\mathscr V}^j$
  with  centers $q^{i,j}$  
  be given for distinct $i,j \in I$; 
  then  the triples 
  $\{ q^{i,j},x_i,y_j  \}$ 
  ($x,y \in \tran$) for $y_j = \xi^{i,j}(x_i)$, $i,j \in I$
  are considered as `perspective rays'.  
  In other words, let $\xi$ be a map 
  $\xi\colon I \times I \longrightarrow S_\tran$
  such that $\xi(i,i) = \id$ and $\xi(j,i) = {(\xi(i,j))}^{-1}$
  for $i,j \in I$; then $\xi^{i,j}(x_i) = (\xi(i,j)(x))_j$
  and $q^{i,j}$ are abstract ``new" points such that
  the perspective rays have form
  $\{ q^{i,j},x_i,\xi^{i,j}(x_i) \}$ for all 
  distinct $i,j \in I$ and $x \in \tran$.
  Finally, let ${\goth H} = \struct{\sub_2(I),\lines'_0}$ be a 
  $\konftyp({\binom{n}{2}},{n-2},{\binom{n}{3}},{3})$-configuration
  and let 
  $\lines_0$ be the image of the family $\lines'_0$ under the map
  $\sub_2(I)\ni \{ i,j \}\longmapsto q^{i,j}$.
  The union of the configurations ${\mathscr V}^i$, $i \in I$,
  the points $q^{i,j}$, $\{ i,j \}\in \sub_2(I)$,
  the perspective rays,
  and the lines $\lines_0$ will be denoted by 
  $\xwlepq({I},{},{\xi},{},{\goth H})$
  and will be called {\em a system of triangle perspectives}.
  The line $L$ is called its {\em  basis}.
  \par
  A system of triangle perspectives is {\em simple} if 
  ${\goth H} = \GrasSpace(I,2)$.
\end{constr}
In the sequel in most parts we shall adopt simply $I = \{ 1,\ldots,n \}$
and we shall consider, mainly, simple \MTC's.

  Let $\rho\in S_{\tran}$ be a ``rotation": $\rho(a) = b$, $\rho(b) = c$,
  $\rho(c) = a$.
  Let us write $\sigma_x$ for the map $\sigma_x\in S_\tran$
  such that $\{ x,y,\sigma_x(y) \} = \tran$
  for any  distinct $x,y \in \tran$.
  Evidently, 
    $\xi(i,j) \in \{ \id,\rho,\rho^{-1},\sigma_x\colon x \in \tran \}$
  for every $i,j \in I$.
It is clear that 
  $\xwlepq({I},{},{\xi},{},{\goth H})$
is a $\konftyp({\binom{n+3}{2}},{n+1},{\binom{n+3}{3}},{3})$-configuration,
where $n = |I|$.
Note that if $n = 3$ then $\goth H$ is the line $\GrasSpace(3,2)$
and thus corresponding \MTC's are simple.
%

\subsection{Main results}

Recall a known classifying theorem:
\begin{thm*}[{see \cite{klik:VC}}]
  Let ${\goth M}$ be a $\konftyp(10,3,10,3)$-configuration
  (i.e. let it be a binomial $\binkonf(n,1)$-configuration with $n=4$).
  Assume that $\goth M$ freely contains at least $2$ ($=n-2$) graphs $K_n$.
  Then either $\goth M$ is the Desargues Configuration $\GrasSpace(5,2)$ with 
  $n+1 = 5$ subgraphs $K_4$, or it is the Kantor $10_3G$-configuration with 
  $n-1 = 3$ subgraphs $K_4$, or it is the fez configuration with $n-2$ subgraphs
  $K_4$. 
  No $10_3$-configuration freely contains a $K_5$-graph.
\end{thm*}
In our investigations the following theorem, which directly follows from a general theorem 
\cite[Thm. 2.20, Cor. 2.21]{klik:binom},
is crucial
\begin{thm*}
  Let $\goth M$ be a $\konftyp(15,4,20,3)$-configuration
  (i.e. let it be a binomial $\binkonf(n,1)$-configuration with $n=5$).
  Assume that $\goth M$ freely contains at least $3$ ($=n-2$) graphs $K_n$.
  Then either $\goth M$ is the combinatorial Grassmannian $\GrasSpace(6,2)$
  with $6 = n+1$ graphs $K_5$, or it is a simple multiveblen configuration on
  $15$ points, with at least $4=n-1$ graphs $K_5$, or it is a (simple) system of triangle
  perspectives on $15$ points. No such a configuration freely contains a $K_6$-graph.
\end{thm*}
On the other hand, the simple multiveblen configurations on $15$ points are all 
known:
\begin{thm*}[{see \cite[Thm. 4]{pascvebl}}]
  Up to an isomorphism, 
  there are exactly three simple multiveblen configurations on $15$ points;
  these are the following:
  \begin{ctext}
    $\xwlepp({J},{p},{K_J},{},{\GrasSpace(J,2)}) \cong \GrasSpace(6,2)$,
    $\xwlepp({J},{p},{N_J},{},{\GrasSpace(J,2)}) \cong \VeblSpace(4)$ 
    (cf. \cite{pascvebl} for a definition of the latter structure),
    and
    $\xwlepp({J},{p},{L_J},{},{\GrasSpace(J,2)})$, 
  \end{ctext}
  where $L_J$ is a linear graph on $J = \{ 1,2,3,4 \}$.
\end{thm*}
The aim of this paper is to prove the following complete classification of all the
$\konftyp(15,4,20,3)$-configurations which freely contain at least three graphs $K_5$:
\begin{thm*}[main]
  There are exactly seventeen binomial partial Steiner triple systems on 15 points
  that freely contain at least three graphs $K_5$ each.
  These are systems 
  $\xwlepq({I},{},{\xi},{},{\GrasSpace(I,2)})$ with $I = \{ 1,2,3 \}$
  of triangle perspectives determined by the following triples 
  $(\xi(1,2),\xi(2,3),\xi(1,3))$:
  \begin{ctext}
    $(\rho,\rho,\rho)$,
    $(\rho,\rho,\id)$,
    $(\rho,\id,\id)$,
    $(\rho,\rho,\rho^{-1})$,
    $(\rho,\rho^{-1},\id)$,
    $(\sigma_x,\sigma_y,\sigma_z)$,
    $(\sigma_x,\sigma_y,\id)$,
    $(\sigma_x,\rho,\id)$,
    $(\sigma_x,\sigma_y,\rho)$,
    $(\sigma_x,\sigma_y,\rho^{-1})$,
    $(\sigma_x,\sigma_x,\sigma_y)$,
    $(\sigma_x,\sigma_x,\rho)$,
    $(\rho,\rho,\sigma_z)$,
    $(\rho,rho^{-1},\sigma_z)$,
    $(\sigma_x,\sigma_x,\sigma_x)$,
    $(\id,\id,\sigma_x)$, and
    $(\id,\id,\id)$
  \end{ctext}
  ($x,y,z$ are arbitrary, with $\{ x,y,z \} = \{ a,b,c \}$).
\end{thm*}
The last three triples determine a {\sf PSTS} with (at least) four graphs $K_5$,
and among them the last one contains six graphs $K_5$.
\par\medskip
Some remarks on specific features of the geometry of corresponding configurations
together with characterizations of their automorphism groups are given in the next
section,
where, in several steps we prove our main result.

\section{Reasoning}

Let us adopt the notation of Subsection \ref{ssec:STP}
We write 
$p\collin p'$ when points $p,p'$ are collinear.

\subsection{Relations between multiveblen configurations and systems of triangle
perspectives} \label{sec:prelimin}

%
Systems of triangle perspectives generalize multiveblen configurations:
{\em every simple multiveblen configuration is a simple system of triangle perspectives}.
More precisely, we have the following
\begin{prop}\label{prop:MVC2MTC:gen}
  Let $I = \{ 1,\ldots,n \}$, $p = \{ n+1,n+2 \}$,
  $\cal P$ be a graph defined on the set $I \cup \{ 0 \} =: I'$,
  and ${\goth H}'$ be a 
  $\konftyp({\binom{n+1}{2}},{n-1},{\binom{n+1}{3}},{3})$-configuration
  defined on the point-set $\sub_2(I')$ such that
  $\left\{\{ 0,i \}, \{ 0,j \}, \{ i,j \}\right\}$ is a line of 
  ${\goth H}'$ for every $\{ i,j \}\in \sub_2(I)$.
  Let ${\goth H}$ be the subconfiguration of ${\goth H}'$ with the point
  set $\sub_2(I)$.
  Then 
  there is a map $\xi \colon I \times I \longrightarrow S_\tran$ 
  such that
  %
    $\xwlepp({I'},{p},{\cal P},{},{{\goth H}'}) \cong
    \xwlepq({I},{},{\xi},{},{\goth H})$.
  %
  \par
  Consequently, for every $\cal P$ as above,
  %
    $\xwlepp({I'},{p},{\cal P},{},{\GrasSpace(I',2)}) \cong
    \xwlepq({I},{},{\xi},{},{\GrasSpace(I,2)})$.
\end{prop}
\begin{proof}
  In the first step we note that there is a graph ${\cal P}'$ such that
  $\{ 0,i \}\in{\cal P}'$ for every $i \in I$ and
    $\xwlepp({I'},{p},{\cal P},{},{{\goth H}'}) \cong
    \xwlepp({I'},{p},{{\cal P}'},{},{{\goth H}'})$.
  For two graphs ${\cal P}_1$, ${\cal P}_2$ on $I'$ and $i_0\in I'$ we write
  ${\cal P}_2 = \mu_{i_0}({\cal P}_1)$ iff two conditions hold:
  $\skros(i_0,j)\in{\cal P}_2$
  iff $\skros(i_0,j)\notin{\cal P}_1$ for $j\neq i_0$, and 
  $\skros(i,j)\in{\cal P}_2$ iff $\skros(i,j)\in{\cal P}_1$
  for $j,i \neq i_0$. Let 
    $\{ i_1,\ldots,i_s \} = \{  i \in I \colon \{ 0,i \}\notin{\cal P} \}$ 
  and
  let ${\cal P}'$ be the image of $\cal P$ under the composition 
  $\mu_{i_1}\circ\ldots\circ\mu_{i_s}$.
  By \cite[Prop. 9]{pascvebl}, 
      $\xwlepp({I'},{p},{\cal P},{},{{\goth H}'}) \cong
    \xwlepp({I'},{p},{{\cal P}'},{},{{\goth H}'})$.
  Define 
  \begin{ctext}
    $\xi(i,j)(c) = c$ \quad and \quad
    $\xi(i,j)(a,b) = \left\{ \begin{array}{ll}
      (a,b)  & \text{if } \{ i,j \}\in {\cal P}'
      \\
      (b,a)  & \text{if } \{ i,j \}\notin {\cal P}'
    \end{array}\right.$;
  \end{ctext}
  for all $\{ i,j \}\in \sub_2(I)$.
  It is seen that the following relabelling 
  \begin{ctext}
    $q^c \mapsto p$, \ $q^b \mapsto a_0$, \ $q^a \mapsto b_0$, \ 
    $a_i \mapsto a_i$, \ $b_i \mapsto b_i$, \ $c_i \mapsto c_{\{ 0,i \}}$
    \ for $i \in I$, 
    \\
    and \  $q^{i,j} \mapsto c_{\{ i,j \}}$ \ for $\{ i,j \}\in\sub_2(I)$
  \end{ctext}
  establishes an isomorphism of 
    $\xwlepq({I},{},{\xi},{},{\goth H})$
  onto 
    $\xwlepp({I'},{p},{{\cal P}'},{},{{\goth H}'})$.
\end{proof}
Let us make some comments on the idea of the proof of \ref{prop:MVC2MTC:gen}.
Note that, in view of \ref{cnstr:MTC},
the role of the points $q^c,q^b,q^a$ on the line $L$ is symmetric.
Observing the proof we note that the point $q^c$ was chosen 
in $\xwlepq({I},{},{\xi},{},{\goth H})$
so that it appears to be the center of the given 
    $\xwlepp({I'},{p},{{\cal P}'},{},{{\goth H}'})$.
But this center can be chosen arbitrary on $L$ and thus, reversing the reasoning
of that proof we obtain
\begin{prop}\label{prop:MTC2MVC:1}
  Let $I,I',p$ be as in \ref{prop:MVC2MTC:gen};
  let $\goth H$ be a
  $\konftyp({\binom{n}{2}},{n-2},{\binom{n}{3}},{3})$-configuration
  on the pointset $\sub_2(I)$.
  Assume that 
  $\xi\colon I\times I \longrightarrow S_{\tran}$ fixes
  some $x \in \tran$ i.e. $\xi^{i,j}(x) = x$ for all $i,j$.
  Then there is a graph $\cal P$ on $I'$ and 
  a $\konftyp({\binom{n+1}{2}},{n-1},{\binom{n+1}{3}},{3})$-configuration
  ${\goth H}'$ extending $\goth H$ such that
    $ \xwlepq({I},{},{\xi},{},{\goth H}) \cong
      \xwlepp({I'},{p},{\cal P},{},{{\goth H}'})$.
  \par
  In particular, for every such $\xi$ there is $\cal P$ with
    $\xwlepq({I},{},{\xi},{},{\GrasSpace(I,2)}) \cong
     \xwlepp({I'},{p},{\cal P},{},{\GrasSpace(I',2)})$.
\end{prop}
However, there are systems of triangle perspectives
that are not multiveblen configurations  
(see \ref{prop:class:MTC3}, \ref{rem:MTC3:MVC}).

\subsection{Systems of triangle perspectives and their characteristic subconfigurations}
\label{sec:subconf}

We start with determining Desargues and Veronese subconfigurations contained in
a given system of triangle perspectives.
In essence, we shall determine such subconfigurations
with one of their lines being the distinguished base line $L$.
Let us begin with some observations.
\begin{repr}\label{repr:10.3inMTC}
  Let us consider a paricular case $I = \{ 1,2 \}$ and
  ${\goth E} = \xwlepq({I},{},{\xi},{},{\goth Z})$,
  where ${\goth Z} = \GrasSpace(I,2)$ is a trivial structure consisting of one single point.
  Clearly, $\goth E$ is a $\left( {10}_{3} \right)$-configuration.
  Then $\xi = \xi(1,2)\in S_\tran$ is of one of the following three types.
  \begin{typcription}
  \typitem{$\xi = \id$}\label{10.3inMTC:1}
    Then $\goth E$ is the classical Desargues configuration \Des. 
  \typitem{$\xi = \sigma_c$}\label{10.3inMTC:2}
    Then $\goth E$ is the Veronese configuration $\VerSpace(3,3)$.
  \typitem{$\xi = \rho$}\label{10.3inMTC:3}
    Then $\goth E$ is 
    another cousin of the \Des-configuration, visualized in Figure \ref{fig:bolfez}.
    We shall call it the \Desr-configuration
    (it is sometimes called the fez configuration, see \cite{klik:VC}).
\end{typcription}
  Note (comp. \cite{klik:VC}) that
  through \eqref{10.3inMTC:1}--\eqref{10.3inMTC:3} we have shown
  all the possible $({10}_{3})$-configurations that can be presented as a perspective
  of two triangles.
  \par\medskip
  Now, let $I$ be arbitrary and 
  ${\goth N} = \xwlepq({I},{},{\xi},{},{\goth H})$.
  {\em For any distinct $i',i''\in I$ the two Veblen subconfigurations 
  ${\mathscr V}^{i'}$, ${\mathscr V}^{i''}$ of $\goth N$ with the common line $L$
  completed by the point $q^{i',i''}$
  yield a $\left( {10}_{3} \right)$-subconfiguration ${\goth N}_{i',i''}$ and
  ${\goth N}_{i',i''} \cong {\goth E}$, where $\goth E$ is one of the three
  introduced through \eqref{10.3inMTC:1}--\eqref{10.3inMTC:3} above};
  namely: 
  \Des\ for $\xi(i',i'') = \id$, 
  \Desv\ for $\xi(i',i'') = \sigma_x$, $x \in\tran$, and
  \Desr\ for $\xi(i',i'') = \rho,\rho^{-1}$.
  \par
  Finally, comparing the Pasch subconfigurations of $\goth M$ and $\goth E$
  we can justify that,
  conversely, 
  {\em if $\goth E$ is a substructure of $\goth N$ such that $L$ is a line
  of $\goth E$ and $\goth E$ is isomorphic to one of \eqref{10.3inMTC:1}--\eqref{10.3inMTC:3}
  then $\goth E$ is spanned by ${\mathscr V}^{i'},{\mathscr V}^{i''}$
  for some $i',i'' \in I$}.
\end{repr}

\begin{figure}[h!]
 \begin{center}
   \begin{tabular}{cc}
     \begin{minipage}[m]{0.4\textwidth}
       \begin{center}
        \includegraphics{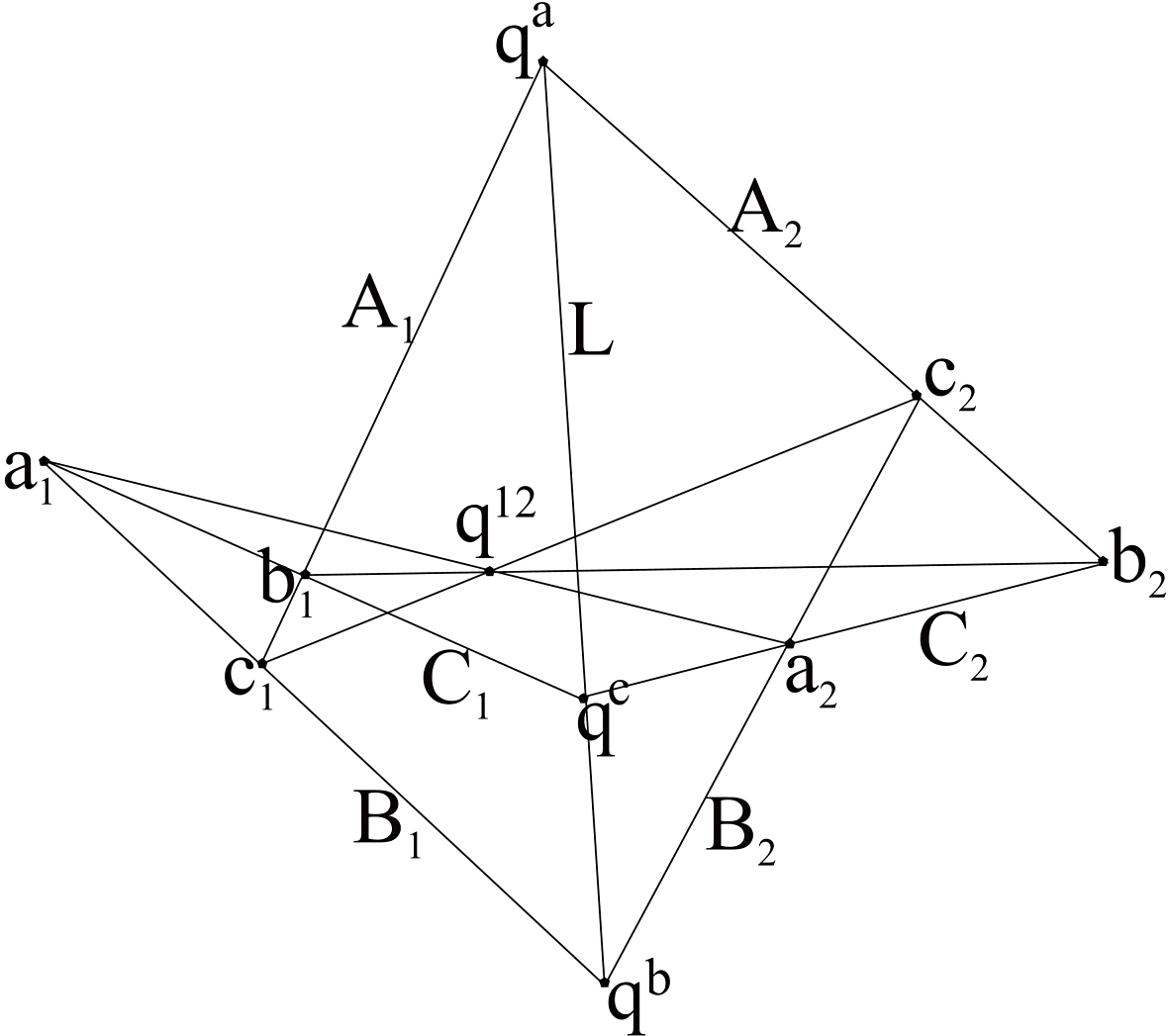}
       \end{center}
       \begin{center}\refstepcounter{figure}\label{fig:boldes}
        Figure \ref{fig:boldes}: The \Des-configuration.
       \end{center}
     \end{minipage}
     &
     \begin{minipage}[m]{0.5\textwidth}
       \begin{center}
        \includegraphics{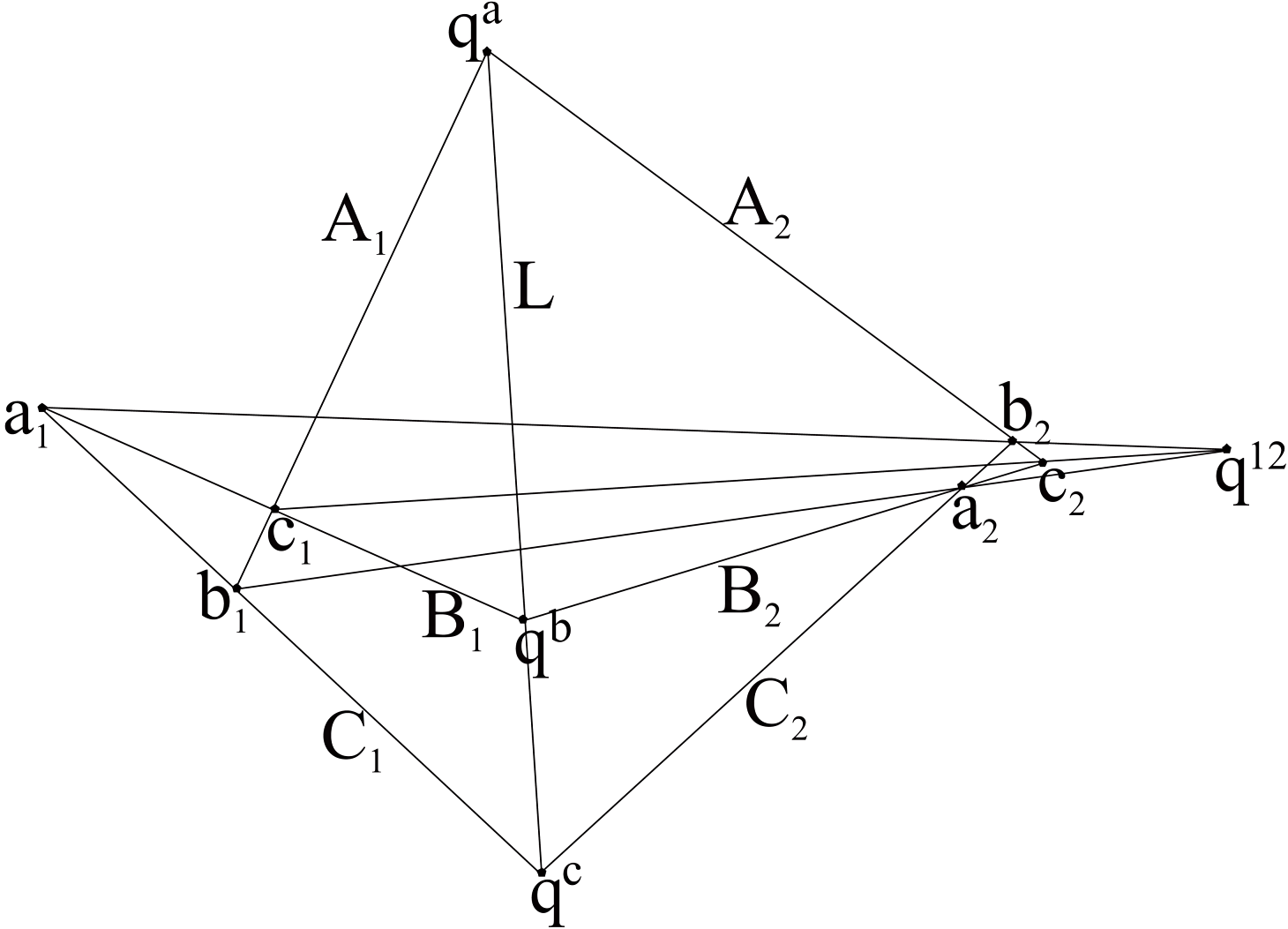}  
       \end{center}
       \begin{center}\refstepcounter{figure}\label{fig:bolkantor}
        Figure \ref{fig:bolkantor}: The \Desv-configuration.
       \end{center}
     \end{minipage}
   \end{tabular}
 \end{center}
\end{figure}

\begin{figure}[h!]
  \begin{center}
  \begin{tabular}{cl}
    \begin{minipage}[m]{0.4\textwidth}
     \begin{center}
       \includegraphics{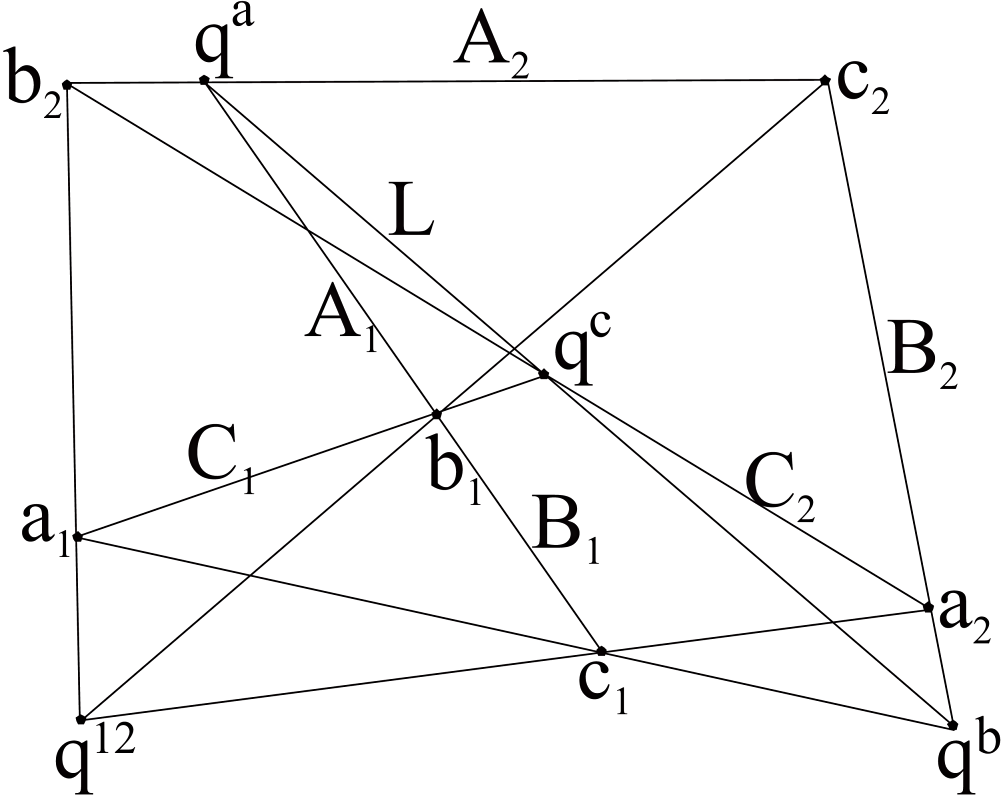}  
     \end{center}
    \end{minipage}
    &
    \begin{minipage}[m]{0.5\textwidth}\refstepcounter{figure}\label{fig:bolfez}
      Figure \ref{fig:bolfez}: The \Desr-configuration.
    \end{minipage}
  \end{tabular}
  \end{center}
\end{figure}

\begin{table}
\begin{center}
\begin{tabular}{cc}
\begin{minipage}[m]{0.4\textwidth}
\begin{center}
  \includegraphics{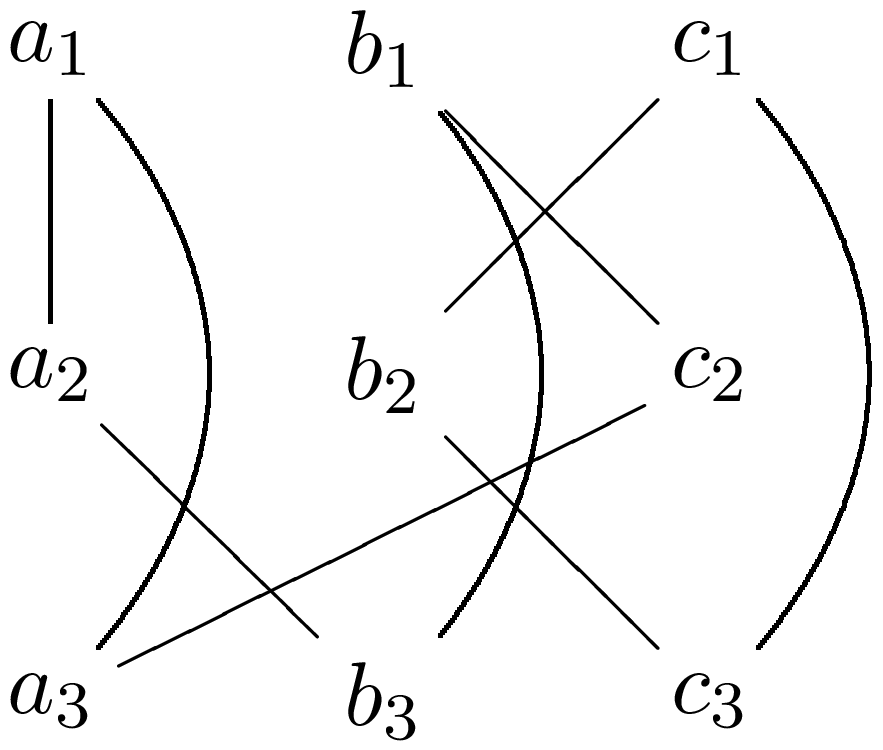}
\end{center}
\end{minipage}
&
\begin{minipage}[m]{0.55\textwidth}
  Points in a row are the points of a triangle of one of the 
  Veblen figures with common line $L$ that are not on $L$;
  points in a column correspond each to other in that they
  are collinear with the same points on $L$.
  Lines indicate that the respective points (from distinct
  triangles) are collinear.
\end{minipage}
\end{tabular}
\end{center}
\caption{The diagram of the line $L$ in the case %
  \eqref{class3:8} of \ref{classif:MTC3}.}
\label{tab:hvl:diagram}
\end{table}

The following is immediate (it can be even stated without
proof, we think), 
though it is quite central in our investigations.
\begin{lem}\label{prop:Lunique} \label{cor:Qunique}
  Let $I = \{ 1,\ldots,m \}$, $\xi\colon I\times I\longrightarrow S_{\tran}$,
  and
  ${\goth M} = \xwlepq({I},{},{\xi},{},{\GrasSpace(I,2)}) = \struct{S,\lines}$
  be an \MTC.
  Then
  \begin{sentences}\itemsep-2pt
  \item
    the set 
	 $${\cal K}_i = \{ a_i,b_i,c_i \} \cup \{ q^{i,j} \colon j \in I\setminus \{ i \} \}$$
	 is a $K_{m+2}$-graph freely contained  in $\goth M$ for each $i\in I$.
  \end{sentences}
  Recall that $\goth M$ is a binomial $\binkonf(m,3)$ partial Steiner triple system,
  so the maximal size of a complete subgraph of $\goth M$ is $n = m+2$.
  \par
  Assume, moreover, that $\goth M$ freely contains exactly $n-2 = m$ graphs $K_n$.
  \begin{sentences}\setcounter{sentence}{1}\itemsep-2pt
  \item
    The family 
	${\mathscr K}= \{ {\cal K}_i\colon i \in I \}$ 
	consists of the complete $K_n$-graphs freely contained in $\goth M$.
  \item
    ${\mathscr Q} := \{ q^{i,j}\colon \{ i,j \}\in\sub_2(I) \}
	= \{ {\cal K}'\cap{\cal K}''\colon  
	{\cal K}',{\cal K}''\in{\mathscr K},\;{\cal K}'\neq{\cal K}''\}$,
	so $\mathscr Q$ remains distinguishable in $\goth M$ as well.
  \item
    $L = \bigcap \{ S\setminus {\cal K}\colon {\cal K}\in{\mathscr K} \}$
    and therefore $L$ is distinguishable in $\goth M$.	
  \item
    Set 
	${\cal D}_x = \{ x_i\colon i \in I \}$ for $x \in \tran$.
	Then 
	${\cal D}_x = \{ a\in S\colon a\not\collin q^x,\; a\notin Q \}$,
	so the family
	${\mathscr D} = \{ {\cal D}_x \colon x \in \tran \}$
	is distinguishable in terms of the structure $\goth M$.
  \end{sentences}
  Consequently, the set $L$ and the families $\mathscr K$, $\mathscr Q$, and 
  $\mathscr D$ remain invariant under the automorphisms of $\goth M$.
\end{lem}

\begin{cor}\label{cor:MTC2MVC:no}
  Assume that
  for every $x \in \tran$ there is $\{ i,j \}\in \sub_2(I)$ such that
  $\xi(i,j)(x) \neq x$ (cf. \ref{prop:MTC2MVC:1}).
  Then the structure 
  $\xwlepq({I},{},{\xi},{},{\GrasSpace(I,2)})$
  is not a simple multiveblen configuration.
\end{cor}
\begin{proof}
  Suppose that 
  ${\goth N} := \xwlepq({I},{},{\xi},{},{\GrasSpace(I,2)})$
  contains one more complete $K_n$-graph $Y$. Here, we use some observations of
  \cite{klik:binom}.
  For each $i\in I$ there is a point $u_i$ such that
  $Y \cap {\cal K}_i = \{ u_i \}$; then $u_i\in\{a_i,b_i,c_i\}$.
  We have $|Y\setminus\{u_i\colon i\in I\}| = 2$, so there are two points $x',y'$ such 
  that $Y \setminus \{ u_i\colon i \in I \} = \{ x',y' \}$, and $x',y'\in L$.
  There are $x,y\in\tran$ such that $\{ x',y' \} = \{ q^x,q^y \}$.
  Take $z\in \tran$ such that 
  $\{ x, y, z \} =\tran$; since $Y$ is a complete graph we have $u_i\collin q^x,q^y$
  and therefore $u_i=z_i$ for every $i\in I$.
  On the other hand the points 
  $Y\cap {\cal K}_i$, $Y\cap{\cal K}_j$, ${\cal K}_i\cap{\cal K}_j$
  are on a line for any two distinct $i,j\in I$. This means that 
  $\{z_i,z_j,q^{i,j}\}$ is a line of $\goth N$ i.e. $\xi(i,j)(z)=z$ for all $i,j$.
  This contradicts the assumptions.
\end{proof}

One more evident observation will be also useful
\begin{lem}\label{lem:isoMTC}
  Let $\alpha\in S_I$, $\beta\in S_\tran$, and $\xi$ be a matrix defining
  a system of triangle perspectives
  $\xwlepq({I},{},{\xi},{},{\goth H})$.
  Let $\goth K$ be the image of $\goth H$ under the map
  $\{ i,j\} \mapsto \{ \alpha(i),\alpha(j) \}$
  defined on $\sub_2(I)$; finally, let a matrix $\zeta$ be defined by
  \begin{ctext}
    $\zeta(i,j) = \beta \circ \xi(\alpha^{-1}(i),\alpha^{-1}(j)) \circ \beta^{-1}$,
    for all $i,j \in I$.
  \end{ctext}
  Then the following map $F$,
  \begin{ctext}
    $F(x_i) = (\beta(x))_{\alpha(i)}$, \ 
    $F(q^x) = q^{\beta(x)}$, \ 
    $F(q^{i,j}) = q^{\alpha(i),\alpha(j)}$, \ 
    $x \in \tran, \; i,j \in I$
  \end{ctext}
  is an isomorphism of 
  $\xwlepq({I},{},{\xi},{},{\goth H})$
  onto
  $\xwlepq({I},{},{\zeta},{},{\goth K})$.
  The map $F$ will be denoted by $\beta\times\alpha$.
\end{lem}
As an immediate consequence we get
\begin{lem}\label{lem:autMTC3:0}\label{lem:autMTC:pom1}
  Let $\alpha\in S_I$ and $\beta\in S_\tran$.
  If the map $F_0\colon x_i\longmapsto \beta(x)_{\alpha(i)}$
  ($x \in \tran$, $i \in I$) preserves the diagram of $L$
  (cf. Table \ref{tab:hvl:diagram}),
  i.e. $\xi(i,j)(x) = y$ iff $\xi(\alpha(i),\alpha(j))(\beta(x)) = \beta(y)$
  for all $x,y \in \tran$, $i,j \in I$
  then  $\beta\times\alpha$ is a (unique) automorphism of 
  $\xwlepq({I},{},{\xi},{},{\GrasSpace(I,2)})$ extending~$F_0$.
 \par
  Conversely,
  if $F\in\Aut({\xwlepq({I},{},{\xi},{},{\GrasSpace(I,2)})})$ 
  preserves the line $L$ then there are $\alpha\in S_I$ and 
  $\beta\in S_\tran$ such that $F = \beta\times\alpha$.
\end{lem}
\begin{proof}
  The first statement is evident in view of \ref{lem:isoMTC}.
  Let $F\in\Aut({\xwlepq({I},{},{\xi},{},{\GrasSpace(I,2)})})$ 
  preserve $L$. So, $F$ determines a permutation $\alpha$ of the Veblen
  subconfigurations of  ${\xwlepq({I},{},{\xi},{},{\GrasSpace(I,2)})}$
  having $L$ as one of its lines defined by the condition
  $F({\mathscr V}^{i}) = {\mathscr V}^{\alpha(i)}$. 
  Then for every $i\in I$ there is $\beta_i\in S_\tran$
  with $F(x_i) = {\beta_i(x)}_{\alpha(i)}$; 
  this gives $F(q^x) = q^{\beta_i(x)}$ for $x \in \tran$.
  Thus $\beta_i = \beta_j =: \beta$ for all $i,j\in I$
  and, finally, $F = \beta\times\alpha$.
\end{proof}

\subsection{Classifications}\label{sec:classif}

\begin{classif}\label{classif:MTC3}
  Let $I = \{ 1,2,3 \}$ and ${\mathscr A} = \{ a_i,b_i,c_i \colon i \in I \}$.
  Recall that $\rho(a,b,c) = (b,c,a)$. Consider the substructure of 
  ${\goth N} := \xwlepq({I},{},{\xi},{},{\GrasSpace(I,2)})$ determined by $\mathscr A$.
  In the following list we indicate types of possible triples
  $\xi(1,2),\xi(2,3),\xi(1,3)$.
  \begin{typcription}
  \typitem{$\rho,\rho,\rho$}\label{class3:1}
    Let $\xi(1,2) = \xi(2,3) = \xi(1,3) = \rho$.
    Then $\mathscr A$ is the $9$-gon
    $( a_1,b_2,c_3,b_1,c_2,a_3,c_1,a_2,b_3 )$.
  \typitem{$\rho,\rho,\id$}\label{class3:2}
    Set $\xi(1,2) = \rho = \xi(2,3)$ and $\xi(1,3) = \id$.
    Then $\mathscr A$ is the $9$-gon
    $( a_1,b_2,c_3,c_1,a_2,b_3,b_1,c_2,a_3 )$.
  \typitem{$\rho,\id,\id$}\label{class3:3}
    Let $\xi(1,2) = \rho$, $\xi(1,3) = \xi(2,3) = \id$.
    Then $\mathscr A$ is the $9$-gon 
    $( a_1,b_2,b_3,b_1,c_2,c_3,c_1,a_2,a_3 )$.
  \typitem{$\rho,\rho,\rho^{-1}$}\label{class3:4}
    Set $\xi(1,2) = \rho = \xi(2,3) = \xi(3,1)$.
    Then $\mathscr A$ consists of three triangles
    $\left( a_1,b_2,c_3 \right)$,
    $\left( b_1,c_2,a_3 \right)$, and
    $\left( c_1,a_2,b_3 \right)$.
  \typitem{$\rho,\rho^{-1},\id$}\label{class3:5}
    Set $\xi(1,2) = \rho = \xi(3,2)$ and $\xi(1,3) = \id$.
    Then $\mathscr A$ consists of three triangles
    $\left( a_1,b_2,a_3 \right)$,
    $\left( b_1,c_2,b_3 \right)$, and
    $\left( c_1,a_2,c_3 \right)$.
  \typitem{$\sigma_x,\sigma_y,\sigma_z$}\label{class3:6}
    Let $\xi(1,2) = \sigma_a$, $\xi(1,3) = \sigma_c$, and $\xi(2,3) = \sigma_b$.
    Then $\mathscr A$ is the union of the hexagon
    $\left( a_2,c_3,c_1,b_2,b_3,a_1 \right)$ and the triangle 
    $\left( c_2,b_1,a_3 \right)$.
  \typitem{$\sigma_x,\sigma_y,\id$}\label{class3:7}
    Set $\xi(1,2)= \sigma_a$, $\xi(2,3) = \sigma_c$, and $\xi(1,3) = \id$.
    Then $\mathscr A$ is the $9$-gon
    $\left( a_1,a_2,b_3,b_1,c_2,c_3,c_1,b_2,a_3 \right)$.
  \typitem{$\sigma_x,\rho,\id$}\label{class3:8}
    Set $\xi(1,2) = \sigma_a$, $\xi(2,3) = \rho$, and $\xi(1,3) = \id$.
    Then $\mathscr A$ is the union of the 
    hexagon $\left( a_1,a_2,b_3,b_1,c_2,a_3 \right)$
    and the triangle $\left( c_1,b_2,c_3 \right)$.
  \typitem{$\sigma_x,\sigma_y,\rho$}\label{class3:9}
    Set $\xi(1,2) = \sigma_a$, $\xi(2,3) = \sigma_c$, and $\xi(1,3) = \rho$.
    Then $\mathscr A$ consists of three triangles
    $\left( a_1,a_2,b_3 \right)$,
    $\left( b_1,c_2,c_3 \right)$, and 
    $\left( c_1,b_2,a_3 \right)$.
  \typitem{$\sigma_x,\sigma_y,\rho^{-1}$}\label{class3:14}
    Set $\xi(1,2) = \sigma_a$, $\xi(2,3) = \sigma_c$, and $\xi(3,1) = \rho$.
    Then $\mathscr A$ is the $9$-gon 
    $\left( a_1,a_2,b_3,c_1,b_2,a_3,b_1,c_2,c_3 \right)$.
  \typitem{$\sigma_x,\sigma_x,\sigma_y$}\label{class3:10}
    Set $\xi(1,2) = \sigma_c = \xi(2,3)$ and $\xi(1,3) = \sigma_a$.
    Then $\mathscr A$ consists of the 
    hexagon $\left( b_1,a_2,b_3,c_1,c_2,c_3 \right)$
    and the triangle $\left( a_1,b_2,a_3 \right)$.
  \typitem{{$\sigma_x,\sigma_x,\rho$}}\label{class3:11}
    Set $\xi(1,2) = \sigma_c = \xi(2,3)$ and $\xi(1,3) = \rho$.
    Then $\mathscr A$ is the $9$-gon
    $\left( a_1,b_2,a_3,c_1,c_2,c_3,b_1,a_2,b_3 \right)$.
  \typitem{{$\rho,\rho,\sigma_z$}}\label{class3:12}
    Set $\xi(1,2) = \rho = \xi(2,3)$ and $\xi(1,3) = \sigma_c$.
    Then $\mathscr A$ consists of the 
    hexagon $\left( a_1,b_2,c_3,c_1,a_2,b_3 \right)$
    and the triangle $\left( b_1,c_2,a_3 \right)$.
  \typitem{{$\rho,\rho^{-1},\sigma_z$}}\label{class3:13}
    Set $\xi(1,2) = \rho$, $\xi(2,3)=\rho^{-1}$ 
    and $\xi(1,3) = \sigma_c$.
    Then $\mathscr A$ consists of the 
    hexagon $\left( a_1,b_2,a_3,b_1,c_2,b_3 \right)$
    and the triangle $\left( c_1,a_2,c_3 \right)$.
  \end{typcription}
  From \ref{cor:MTC2MVC:no}
  we get that in every of the above cases
  $\goth N$ is not a simple multiveblen configuration.
  The numbers of \Des, \Desv, and \Desr-configurations containing $L$ and 
  contained in respective \MTC's,
  together with other important parameters are presented in Table \ref{tab:MTC3}.
\end{classif}
\begin{prop}\label{prop:class:MTC3}
  Let $|I| = 3$.
  The list \eqref{class3:1}--\eqref{class3:13} of \ref{classif:MTC3}
  exhausts all the possible systems $\xi$ 
  (up to permutations of $I$ and of $\tran$)
  such that 
    $\xwlepq({I},{},{\xi},{},{\GrasSpace(I,2)})$ 
  is not a simple multiveblen configuration
  (cf. \ref{prop:MTC2MVC:1}, \ref{cor:MTC2MVC:no}).
 \par
  The structures in this list are pairwise nonisomorphic.
  Consequently, every $\konftyp(15,4,20,3)$-system of triangle perspectives that
  is not a simple multiveblen configuration is isomorphic to exactly one of the structures
  defined through \eqref{class3:1}--\eqref{class3:13} of \ref{classif:MTC3}.
\end{prop}
\begin{proof}
  Any system $\xi$ is determined by three maps: $\xi(1,2),\xi(2,3),\xi(1,3)$,
  each one in one of the following classes 
  $P = \{  \rho,\rho^{-1} \}$, $\Sigma = \{ \sigma_a\sigma_b,\sigma_c \}$,
  and ${\Delta} = \{ \id \}$.
  Thus (cf. \ref{lem:isoMTC})
  the type of $\xi$ can be represented by a 3-multiset with elements in the above 
  three classes. Let us take into account three observations. 
  Firstly, the type $3\times P$
  represents both $2\times\rho+\rho^{-1}$ and $3\times\rho$.
  Secondly, $3\times\Sigma$ represents  $\sigma_{x_1},\sigma_{x_2},\sigma_{x_3}$
  where either the $x_i$ are distinct or some of them coincide 
  (similar remark concerns $2\times P$ and $2\times\Sigma$).
  Thirdly, $\sigma_x\sigma_y = \rho$ or $\sigma_x\sigma_y = \rho^{-1}$;
  the arising configurations are distinguishable by 
  \ref{cor:Qunique} (look at the diagram of $L$ in corresponding structures).
  Then canceling types which, in view of \ref{prop:MTC2MVC:1} produce a
  multiveblen configuration (e.g. $(\sigma_x,\sigma_x,\id)$, which fixes $x$) we obtain
  the list considered in \ref{classif:MTC3}.
  \par
  We write ${\goth N}_{\sf n}$ for the structure defined in (n) of \ref{classif:MTC3},
  with ${\sf n} = 1,...,14$.
  To prove that these structures  
  are pairwise nonisomorphic note, first, that if two of them 
  would be isomorphic then the numbers of their  
  \Des, \Desv, and \Desr-configurations having $L$ as 
  one of its lines should agree.
  If these numbers agree, one compares the type of polygons $\mathscr A$
  (i.e. the diagrams of $L$).
  The above procedure doses not distinguish structures in the following three pairs only:
  (\eqref{class3:6},\eqref{class3:10}), (\eqref{class3:14},\eqref{class3:11}),
  and
  (\eqref{class3:12},\eqref{class3:13}).
 \par
  Let us consider the pair 
  ${\goth N}_{\text{\ref{class3:12}}},\;{\goth N}_{\text{\ref{class3:13}}}$.
  Suppose that $F$ is an isomorphism of ${\goth N}_{\text{\ref{class3:13}}}$ onto 
  ${\goth N}_{\text{\ref{class3:12}}}$.
  Then $F$ maps the line $L$ onto itself.
  Moreover, $F$ maps the  
  triangle (unique in the diagram of $L$) $\Delta_1 = (c_1,a_2,c_3)$ of 
  ${\goth N}_{\text{\ref{class3:13}}}$
  onto the triangle (analogously the unique one) 
  $\Delta_2 = (b_1,c_2,a_3)$ of ${\goth N}_{\text{\ref{class3:12}}}$.
  The triangle $\Delta_2$ crosses each set in the family ${\mathscr D}$ defined
  over ${\goth N}_{\text{\ref{class3:12}}}$,
  while $\Delta_1$
  misses ${\cal D}_b$ of ${\goth N}_{\text{\ref{class3:13}}}$.
  The respective families $\mathscr D$ should be preserved by isomorphisms,
  so no  isomorphism may map $\Delta_1 $ onto $\Delta_2$ and thus
  ${\goth N}_{\text{\ref{class3:12}}}\not\cong{\goth N}_{\text{\ref{class3:13}}}$.
  \par
  Let us consider the pair 
  ${\goth N}_{\text{\ref{class3:6}}},\;{\goth N}_{\text{\ref{class3:10}}}$.
  Suppose that $F$ is an isomorphism of the respective structures. Then $F$
  maps the triangle $\Delta_1 = (c_2,b_1,a_3)$ of ${\goth N}_{\text{\ref{class3:6}}}$ onto
  the triangle $\Delta_2 = (a_1,b_2,a_3)$ of ${\goth N}_{\text{\ref{class3:10}}}$. 
  The reasoning as above shows, that no isomorphism may map $\Delta_1$ onto $\Delta_2$
  and, consequently, 
  ${\goth N}_{\text{\ref{class3:6}}} \not\cong {\goth N}_{\text{\ref{class3:10}}}$.
  \par
  Finally,
  let us consider the pair 
  ${\goth N}_{\text{\ref{class3:14}}},\;{\goth N}_{\text{\ref{class3:11}}}$.
  The $9$-gon $\mathscr A$ of ${\goth N}_{\text{\ref{class3:11}}}$ contains three
  consecutive points in one element of the family $\mathscr D$ 
  (namely: the sequence $(c_1,c_2,c_3)$ in ${\cal D}_c$),
  while no such subsequence can be found in the $9$-gon around $L$ in 
  ${\goth N}_{\text{\ref{class3:14}}}$. Therefore
  ${\goth N}_{\text{\ref{class3:14}}} \not\cong {\goth N}_{\text{\ref{class3:11}}}$.
\end{proof}

\begin{table}[h!]
\small
\begin{center}
\begin{tabular}{p{0.25\textwidth}{l}lcccccccccccccc}
\strut  & \multicolumn{13}{c}{the type in the classification of \ref{classif:MTC3}}
\\
the number of & \ref{class3:1} & \ref{class3:2} & \ref{class3:3} & 
\ref{class3:4} & \ref{class3:5} & \ref{class3:6} & \ref{class3:7} & 
\ref{class3:8} & \ref{class3:9} & \ref{class3:14} & \ref{class3:10} & 
\ref{class3:11} & \ref{class3:12} & \ref{class3:13}
\\ \hline
{\tiny \Des-subconfigurations}  & 0 & 1 & 2 & 0 & 1 & 0 & 1 & 1 & 0 & 0  & 0  & 0  & 0  & 0
\\
{\tiny \Desv-subconfigurations} & 0 & 0 & 0 & 0 & 0 & 3 & 2 & 1 & 2 & 2 & 3 & 2  & 1  & 1
\\
{\tiny \Desr-subconfigurations} & 3 & 2 & 1 & 3 & 2 & 0 & 0 & 1 & 1 & 1 & 0 & 1  & 2  & 2
\\ \hline
{\tiny triangles around $L$} & 0 & 0 & 0 & 3 & 3 & 1 & 0 & 1 & 3 & 0  & 1  & 0  & 1 & 1
\\
{\tiny hexagons around $L$}  & 0 & 0 & 0 & 0 & 0 & 1 & 0 & 1 & 0 & 0 & 1  & 0  & 1 & 1
\\
{\tiny $9$-gons around $L$}  & 1 & 1 & 1 & 0 & 0 & 0 & 1 & 0 & 0 & 1 & 0  & 1  & 0 & 0
\end{tabular}
\end{center}
\caption{Parameters of $\konftyp(15,4,20,3)$-systems of triangle perspectives.}
\label{tab:MTC3}
\end{table}

\begin{rem}\label{rem:MTC3:MVC}
  Note that the three $\konftyp(15,4,20,3)$ simple multiveblen configurations that are 
  systems of triangle perspectives can be 
  associated (as in \ref{classif:MTC3}) with the following triples
  $(\xi(1,2),\xi(2,3),(\xi(1,3)))$
  ($J = \{ 1,2,3,4 \}$):  
  \begin{typcription}
  \typitem{$\id,\id,\id$} 
    $\xwlepp({J},{p},{K_J},{},{\GrasSpace(J,2)}) \cong \GrasSpace(6,2)$,
  \typitem{$\sigma_x,\sigma_x,\sigma_x$}  
    $\xwlepp({J},{p},{N_J},{},{\GrasSpace(J,2)}) \cong \VeblSpace(4)$ 
    (cf. \cite{pascvebl} for a definition of the latter structure),
  \typitem{$\sigma_x,\sigma_x,\id$, or $\id,\id,\sigma_x$}
    $\xwlepp({J},{p},{L_J},{},{\GrasSpace(J,2)})$, where $L_J$ is a linear graph on $J$.
  \end{typcription}
\end{rem}
\begin{rem}\label{exm:strange:1}
  Let $I = \{ 1,2,3 \}$, $\xi(1,2)= \id$, $\xi(2,3) = \rho$, and $\xi(1,3) = \rho$
  (cf. \ref{classif:MTC3}\eqref{class3:5}).
  Then $\xwlepq({I},{},{\xi},{},{\GrasSpace(I,2)}) \cong
  \xwlepp({\mathscr Z},{q^{1,2}},{K_{\mathscr Z}},{},{\goth V})$, 
  where 
  ${\mathscr Z} = {\tran\cup\{ 0 \}}$
  and the Veblen configuration $\goth V$ defined on $\sub_2({\mathscr Z})$ has the lines
  \begin{ctext}
  $\sub_2(\tran)$,
  $\{ \{ 0,a \}, \{ 0,b \}, \{ b,c \} \}$,
  $\{ \{ 0,b \}, \{ 0,c \}, \{ a,c \} \}$, and
  $\{ \{ 0,c \}, \{ 0,a \}, \{ a,b \} \}$.
  \end{ctext}
  Clearly, $\xwlepq({I},{},{\xi},{},{\GrasSpace(I,2)})$
  is not any {\em simple} multiveblen configuration.
\end{rem}
\begin{proof}
  Let us write in the definition of a multiveblen configuration
  $j_a$ instead of $a_j$, and $j_b$ instead of $b_j$ ($j$ is an element of 
  the ``indexing'' set).
  Next, replace 
  the symbols $a,b$ used as indices
  by the symbols `$1$' and `$2$' resp.  
  (in particular, this operation relabels
  $(a_a,b_a,a_b,b_b,a_c,b_c) \mapsto (a_1,a_2,b_1,b_2,c_1,c_2)$).
  Finally, it suffices to denote $p = q^{1,2}$, $0_1 = q^{1,3}$, $0_2 = q^{2,3}$,
  write down the points on lines through $p$,
  and apply definitions of a multiveblen configuration and \ref{cnstr:MTC}.
  Note two examples: the triples $(b_1,c_1,q^a)$ and $(b_2,c_2,q^a)$ are collinear
  and thus $q^a = \kros(b,c)$; the triples $(q^{1,3},c_1,a_3)$ and $(q^{2,3},c_2,a_3)$
  are collinear and thus $a_3 = \kros(0,c)$; and so on.
\end{proof}
%
%
%
%
\begin{rem}\label{rem:MTC3:MVC:nsl}\normalfont
  Note that a multiveblen configuration has a center i.e. a point $p$ such that
  any two lines through $p$ yield a Veblen configuration.
  Considering the lists of \ref{classif:MTC3} and \ref{rem:MTC3:MVC}
  and taking into account \ref{exm:strange:1} we get that
  {\em exactly four $\konftyp(15,4,20,3)$-multiveblen configurations
  can be presented as a system of triangle perspectives}:
  those quoted in \ref{rem:MTC3:MVC} and the one defined in
  \ref{classif:MTC3}\eqref{class3:5}.
  Since there are more than four $\konftyp(15,4,20,3)$-multiveblen configurations
  (cf. \cite{twistfan}) we get that
  {\em not every multiveblen configuration can be presented as a 
  system of triangle perspectives}.
\end{rem}

\begin{rem}\label{rem:MTC3:ver}
  It is known that every Veronese configuration $\VerSpace(3,k)$
  freely contains exactly three $K_{k+1}$-graphs.
  In particular, $\VerSpace(3,4)$ is a $\konftyp(15,4,20,3)$-configuration
  with three $K_5$-graphs, so it is an \MTC.
  Indeed, 
  $\VerSpace(3,4)$ is isomorphic to
  the structure defined in \ref{classif:MTC3}\eqref{class3:6}.
\end{rem}
\begin{proof}
  Let ${\goth V} = \VerSpace(\tran,4)$; set $I = \{ 1,2,3 \}$,
  $\xi(1,2)= \sigma_a$, $\xi(2,3) = \sigma_c$, and $\xi(1,3) = \sigma_b$.
  The following relabelling establishes an isomorphism of 
  ${\goth V}$ and $\xwlepq({I},{},{\xi},{},{\GrasSpace(I,2)})$:
  $abc \tran \mapsto L$,
  $a^2bc \mapsto q^a$, $ab^2c\mapsto q^b$, $abc^2 \mapsto q^c$,
  $ab^3\mapsto a_1$,
  $a^3b\mapsto b_1$,
  $a^2b^2\mapsto c_1$,
  $ac^3\mapsto a_2$,
  $a^2c^2\mapsto b_2$,
  $a^3c\mapsto c_2$,
  $b^2c^2\mapsto a_3$,
  $bc^3\mapsto b_3$,
  $b^3c\mapsto c_3$,
  $a^4\mapsto q^{1,2}$,
  $c^4\mapsto q^{2,3}$,
  $b^4\mapsto q^{1,3}$,
  ${\tran}^4\mapsto Q$.
\end{proof}
%


\subsection{Automorphisms}\label{sec:auty}

In view of \ref{lem:autMTC:pom1} and \ref{prop:Lunique}, 
if  the structure $\xwlepq({I},{},{\xi},{},{\GrasSpace(I,2)})$
is not a simple multiveblen configuration
to determine its automorphism group
it suffices to determine automorphisms of the diagram of $L$
(suitable permutations of its rows and columns).

\begin{prop}\label{prop:autofallbol}
   Let ${\goth N} = \xwlepq({I},{},{\xi},{},{\GrasSpace(I,2)})$ be one of the structures 
   defined in \ref{classif:MTC3}. 
   We write $P$ for the $C_3$-group generated by $\rho$. If 
   $i \in I = \{1,2,3\}$ we write
   $\sigma_i$ for the transposition of the elements in $I\setminus\{i\}$ -- similarly we write 
   $\varrho$
   for the cycle $(2,3,1)$ and $R$ for the group generated by $\varrho$. 
   The group $\goth G$ of automorphisms of $\goth N$ consists of all the 
   maps $\beta \times \alpha$ where $\alpha\in S_{I}$ and $\beta\in S_\tran$
   such that:
   \begin{description}\itemsep-2pt
   \item[{\normalfont cases \eqref{class3:1}, \eqref{class3:2}, \eqref{class3:3}}:]\quad
     $\alpha = \id$, $\beta\in P$ or $\alpha = \sigma_{i_0}$, $\beta = \sigma_x$,
     $x \in \tran$,
     where in the corresponding cases $i_0 = 2$, $i_0 = 2$, and $i_0 = 3$;
     \quad then ${\goth G} \cong C_2\ltimes C_3 = S_3$.
   \item[{\normalfont case \eqref{class3:4}}:]\quad
     $\alpha \in R$, $\beta \in P$
     or 
     $(\alpha,\beta) = (\sigma_i,\sigma_x)$ with 
     $x\in \tran$, $i \in I$; \quad
     then ${\goth G} \cong C_2 \ltimes (C_3\oplus C_3)$.
   \item[{\normalfont case \eqref{class3:5}}:]\quad
     $\alpha\in \{\id, \sigma_2\}$, $\beta\in P$; \quad then ${\goth G} \cong C_2 \oplus C_3$.
   \item[{\normalfont case \eqref{class3:6}}:]\quad
     $(\alpha,\beta) = (\varrho^m,\rho^m)$ with $m = 0,1,2$
     or 
     $(\alpha,\beta) = (\sigma_i,\sigma_x)$ with 
     $(x,i)\in \{ (b,1),(c,2),(a,3) \}$; \quad
     then ${\goth G} \cong S_{\tran} = S_3$
     (cf. \ref{rem:MTC3:ver} and \cite[Prop. 5]{pascvebl}).
   \item[{\normalfont cases \eqref{class3:7}, \eqref{class3:9}, %
    \eqref{class3:14}, \eqref{class3:11},  \eqref{class3:12}}:]\quad
     $(\alpha,\beta) = (\id,\id)$ or
     $(\alpha,\beta) = (\sigma_2,\sigma_{x_0})$,
     where in the corresponding cases
     $x_0 = b$, $x_0 = b$, $x_0 = b$, $x_0 = c$, and $x_0 = c$; \quad 
     then ${\goth G} \cong C_2$.
   \item[{\normalfont case \eqref{class3:8}}:]\quad
     $\alpha = \id$, $\beta = \id$.
   \item[{\normalfont cases \eqref{class3:10}, \eqref{class3:13}}:]\quad
     $\alpha \in \{ \id,\sigma_2 \}$, $\beta = \id$; \quad then ${\goth G} \cong C_2$.
   \end{description}
\end{prop}
 \begin{proof} 
  It is easy to check that in every one of the cases \eqref{class3:1}-\eqref{class3:13} 
  given  maps yield automorphisms of the diagram of $L$. 
  We must verify that $\goth N$ has no other automorphisms.
\par
We say that a set of points $\{a_1,\ldots,a_{n+1}\}$ is a {\em path of length $n$} if 
points $a_i, a_{i+1}$ are  collinear for all $i=1,\ldots,n$ 
and no three consecutive points from this set are on a line. 
\par
  Let $F = \beta\times\alpha \in \Aut({\goth N})$.
  Recall that 
  $F({\mathscr V}^{i}) = {\mathscr V}^{\alpha(i)}$.
  Moreover, if $F(X_i) = Y_j$ then $F(x_i) = y_j$.
  Let us examine corresponding cases of \ref{classif:MTC3}.
%
 \begin{sentences}\itemsep-2pt
   \item[{\normalfont \eqref{class3:1}}:]\quad
      Clearly, $F$
      preserves the $9$-gon $\mathscr A$ in the diagram of $L$;
      %
      %
      let $f$ be the induced automorphism of $\mathscr A$.
      If $f$ is the rotation $\tau_1$ of $\mathscr A$ on $1$ item, then 
      $f(a_1) = b_2$, $f(a_3) = c_1$, so, inconsistently,
      $c = \beta(a) = b$.
      If $f$ is the rotation $\tau_2$ on $2$ items then 
      $f(a_1) = c_3$, $f(a_3) = a_2$ and again a contradiction arises.
      Finally, only the rotation $\tau_3$ of $\mathscr A$ on $3$ items and, consequently, 
      the rotation $\tau_6$ on $6$ items can be written in the form $\beta\times\alpha$
      and then $\alpha = \id$, $\beta \in P$.
      Now, let $f$ be a reflection in a point $d$ of $\mathscr A$.
      Analyzing   $\mathscr A$  we obtain
      $d = a_2$, $d = b_2$ or $d = c_2$, for $\beta=\sigma_a,\sigma_b,\sigma_c$ respectively,
       and $\alpha=\sigma_2$. 
   \item[{\normalfont \eqref{class3:2}}:]\quad
     First, note that $\alpha(2)=2$ since there is only one \Des-configuration
     that contains $L$ and this one contains
     ${\mathscr V}^1$ and ${\mathscr V}^3$.
     Therefore $F$ leaves the set $H:=\{ a_2,b_2,c_2  \}$ invariant.
     Consider the $9$-gon $\mathscr A$ in the diagram of $L$; 
     %
    an automorphism of $\mathscr A$ that  preserves simultaneously $H$
    is either the rotation on $3$ or $6$ items ($b_2\mapsto a_2$ or $b_2 \mapsto c_2$)
    or the reflection in one of the points in $H$.
    This gives the claim.
  \item[{\normalfont \eqref{class3:3}}:]\quad
     Note that 
     $A_3,B_3,C_3$ are the unique lines distinct from $L$, which are contained in both of
     \Des-subconfigurations of $\goth N$. 
     Therefore, $F$ preserves these lines, so $\alpha(3) = 3$.
     Moreover, $F$ yields an automorphism $f$ of the $9$-gon $\mathscr A$
     %
     i.e. its suitable rotation or symmetry; taking into account the fact that
     the set $\{ a_3,b_3,c_3 \}$ must be invariant under $f$ we get the claim.
   \item[{\normalfont \eqref{class3:4}}:]\quad
     It suffices to show that 
     the maps 
       $f_1 = \sigma_c\times \id$ and $f_2 = \id\times\sigma_2$ 
     do not preserve the diagram of $L$.
     Indeed, 
     $b_1 \collin c_2$ and $f_1(b_1) = a_1 \ncollin c_2 = f_1(c_2)$,
     and
     $b_1 \collin a_3$ and $f_2(b_1) = b_3 \ncollin a_1 = f_2(a_3)$.
   \item[{\normalfont \eqref{class3:5}}:]\quad
     The unique \Des-subconfiguration of $\goth N$ containing $L$ contains 
     ${\mathscr V}^1$ and ${\mathscr V}^3$ 
     and thus $F$ maps ${\mathscr V}^2$ onto itself; 
     therefore $\alpha(2)=2$ and $\alpha = \id,\sigma_2$. 
     Finally, note that 
     for every $x \in \tran$ the diagram of $L$ contains the triangle $(x_1, x_3, \rho(x)_2)$
     and no triangle $(x_1,x_3,y_2)$ with $y \neq \rho(x)$;
     consequently, $\beta \neq \sigma_x$.
   \item[{\normalfont \eqref{class3:6}}:]\quad
     It suffices to note that $F$ leaves the triangle
     $(b_1,c_2,a_3)$, the unique one in the diagram of $L$, invariant.
   \item[{\normalfont \eqref{class3:7}}:]\quad
     Note that the set ${\cal D}_x\cap {\mathscr A}$ is 
     either a path of length $2$ if $x=a,c$,
     or a union of a path of length $1$ and the set $\{b_2\}$ if $x=b$. 
     Thus $\beta(b)=b$, and also $F(b_2)=b_2$, which 
     gives $\alpha(2)=2$. 
     The point $a_1$ is the center of the path ${\cal D}_a\cap {\mathscr A}$, and
     $c_3$ is the center of the path ${\cal D}_c\cap {\mathscr A}$. 
     Hence, $F$ fixes each of these two centers or interchanges them.
     Finally we get that either $F=\sigma_b\times\sigma_2$ or $F=\id\times\id$.
   \item[{\normalfont \eqref{class3:8}}:]\quad
     Considering substructures of $\goth N$ spanned by ${\mathscr V}^{i}\cup{\mathscr V}^{j}$
     (that three are pair wise non isomorphic) we get that $\alpha = \id$.
     Then, since $F$ leaves the  triangle $(c_1,b_2,c_3)$ invariant, 
     we conclude with $\beta = \id$.
   \item[{\normalfont \eqref{class3:9}}:]\quad
     The unique \Desr-subconfiguration of $\goth N$ containing $L$ contains
     ${\mathscr V}^1$ and ${\mathscr V}^3$ 
     and thus 
     $\alpha(2) = 2$.
     The unique triangle in the diagram of $L$ which contains a point of each of $D_x$ 
     is $(c_1,b_2,a_3)$, so $F$ preserves it. 
     Particularly, $F(b_2) = b_2$ and the set $\{c_1,a_3\}$ remains invariant under $F$. 
     Finally we get that either $F=\sigma_b\times\sigma_2$ or $F=\id\times\id$.
   \item[{\normalfont \eqref{class3:14}}:]\quad
     As in \eqref{class3:9} we obtain $\alpha(2) = 2$.
     The set ${\cal D}_x\cap {\mathscr A}$ contains no path only for $x=b$, so $\beta(b)=b$.
      Since a path $\Gamma$ 
      is contained in a set 
      ${\cal D}_x\cap {\mathscr A}$,   $x\in\tran$
      exactly when 
      $\Gamma=\{a_1,a_2\}$ or $\Gamma=\{c_2,c_3\}$, 
      either the map $F$ preserves each of these paths and then $F=\id\times\id$, or
      $F$ interchanges them, that gives $F=\sigma_b\times\sigma_2$.
   \item[{\normalfont \eqref{class3:10}}:]\quad
     The set ${\cal D}_x$ has two, one,  and none common elements with the unique triangle
     (the triangle $\{a_1,b_2,a_3\}$) 
     in the diagram of $L$
     exactly when $x=a$, $x=b$, and $x=c$, respectively. 
     Therefore, 
     we get $\beta=\id$. It means also that, in particular, $F(b_2)=b_2$, so 
     $\alpha = \id,\sigma_2$.
   \item[{\normalfont \eqref{class3:11}}:]\quad
     The set ${\cal D}_x\cap {\mathscr A}$ is a path of length $2$ only for $x=c$, 
     thus $\beta(c)=c$ and $F({\mathscr D}_c)={\mathscr D}_c$. 
     The point $c_2$ is  
     the center of this path, so $F(c_2)=c_2$, and then $\alpha(2)=2$.
     The unique points from ${\mathscr A}\setminus {\cal D}_c$ that are collinear 
     with a point of ${\cal D}_c$ are 
     $a_3,b_1$. Therefore the set $\{a_3,b_1\}$ remains invariant under $F$.
     Consequently, either $F=\sigma_c\times\sigma_2$ or $F=\id\times\id$.
   \item[{\normalfont \eqref{class3:12}}:]\quad
     The unique \Desv-subconfiguration of $\goth N$ having $L$ as its line contains 
     ${\mathscr V}^1$ and ${\mathscr V}^3$, so $\alpha(2) = 2$.
     Consider the unique triangle $(b_1,c_2,a_3)$ in the diagram of $L$;
     in view of the above, $F(c_2) = c_2$ which gives $\beta(c) = c$.
     Thus either $F$ fixes $b_1,a_3$ and then $\beta = \id$ and $\alpha = \id$,
     or $F$ interchanges the points $b_1,a_3$, which gives $\beta = \sigma_c$
     and $\alpha = \sigma_2$.
   \item[{\normalfont \eqref{class3:13}}:]\quad
     Consider the unique triangle $\Delta=(c_1,a_2,c_3)$ in the diagram of $L$;
     clearly, $|\Delta\cap{\mathscr D}_a|=1$, $|\Delta\cap{\mathscr D}_c|=2$.
     Since $F$ preserves $\Delta$, $\beta(a)=a$, $\beta(c)= c$, so $\beta = \id$.
     Therefore, $F(a_2)=a_2$.
     Consequently, $\alpha(2) = 2$, so $\alpha = \id,\sigma_2$.
 \end{sentences}
  In view of \ref{lem:autMTC:pom1} and \ref{prop:Lunique} this closes the proof.
 \end{proof}


\bigskip

\par\noindent\small
Authors' address:\\
Krzysztof Petelczyc, Ma{\l}gorzata Pra{\.z}mowska, Krzysztof Pra{\.z}mowski\\
Institute of Mathematics, University of Bia{\l}ystok\\
ul. Akademicka 2\\
15-246 Bia{\l}ystok, Poland\\
e-mail: {\ttfamily kryzpet@math.uwb.edu.pl}, {\ttfamily malgpraz@math.uwb.edu.pl},
{\ttfamily krzypraz@math.uwb.edu.pl}


\end{document}